\documentclass{article}
\usepackage[english]{babel}

\makeatletter
\g@addto@macro\bfseries{\boldmath}
\makeatother

\usepackage{enumerate}
\usepackage[shortlabels]{enumitem}
\usepackage{verbatim}
\usepackage[a4paper,left=3cm,right=3cm]{geometry}

\usepackage{amsmath,amssymb,mathrsfs,amsthm}
\usepackage{graphicx}
\usepackage{xurl}
\usepackage[colorlinks=true, allcolors=blue]{hyperref}
\usepackage{seqsplit}

\usepackage{xcolor}
\usepackage{listings}

\usepackage{microtype}
\usepackage[T1]{fontenc}

\usepackage{authblk}

\definecolor{codebgcolour}{rgb}{0.94,0.94,0.96}
\definecolor{codegreen}{rgb}{0,0.5,0}
\definecolor{codemagenta}{rgb}{0.9,0,0.5}

\lstdefinestyle{betterlst}{
	backgroundcolor=\color{codebgcolour},
	commentstyle=\color{codegreen},
	keywordstyle=\color{codemagenta},
	language=Python,
	numberstyle=\ttfamily\tiny,
	basicstyle=\ttfamily\footnotesize\linespread{1}\selectfont,
	breakatwhitespace=false,         
	breaklines=true,       
	keepspaces=true,                 
	numbers=left,                    
	numbersep=5pt,                  
	showspaces=false,                
	showstringspaces=false,
	showtabs=false,                  
	tabsize=2,
	frame=single
}

\newcommand{\spun}{\operatorname{spun}}
\newcommand{\tri}{\mathscr{T}}
\newcommand{\twoknot}{\mathscr{K}}
\newcommand{\lk}{\operatorname{lk}}
\newcommand{\R}{\mathbb{R}}

\newcommand{\is}[1]{{\noindent\ttfamily\seqsplit{#1}}} 

\usepackage[colorinlistoftodos]{todonotes}
\usepackage{cleveref}
\usepackage{booktabs}

\theoremstyle{plain}
\newtheorem{conjecture}{Conjecture}[section]
\newtheorem{theorem}[conjecture]{Theorem}

\theoremstyle{remark}
\newtheorem{remark}[conjecture]{Remark}

\title{Triangulating Spun 2-Knot Complements}

\title{Triangulating Spun $2$-Knot Complements}

\author[1]{Rhuaidi Antonio Burke}
\author[2]{Benjamin Burton}
\author[3]{Arunima Ray}
\author[4]{Jonathan Spreer}
\author[2]{Finn Thompson}
\author[5]{Orion Zymaris}

\affil[1]{
	{\itshape{Mathematical Institute, University of Oxford, Oxford OX2 6GG, United Kingdom}}
	
	{\ttfamily{rhuaidi.burke@maths.ox.ac.uk}}
}

\affil[2]{
	{\itshape{School of Mathematics and Physics, The University of Queensland, Brisbane, Queensland 4072, Australia}}
	
	{\ttfamily{bab@maths.uq.edu.au}}\\
	{\ttfamily{f.thompson@uq.net.au}}
}

\affil[3]{
	{\itshape{School of Mathematics and Statistics, University of Melbourne, Parkville, Victoria 3010, Australia}}
	
	{\ttfamily{aru.ray@unimelb.edu.au}}
}

\affil[4]{
	{\itshape{School of Mathematics and Statistics, The University of Sydney, Sydney, New South Wales 2006, Australia}}
	
	{\ttfamily{jonathan.spreer@sydney.edu.au}}
}

\affil[5]{
	{\itshape{School of Mathematics, Monash University, Victoria 3800, Australia}}
	 
	{\ttfamily{orion.zymaris@monash.edu}}
}
\normalsize

\date{\today}

\begin{document}
	
	\maketitle
	
	\begin{abstract}
		A $2$-knot is an embedding of a $2$-sphere into the $4$-sphere. Similar to the case of embedding circles into the $3$-sphere, this allows the $2$-sphere to be knotted. In this short paper, we present an algorithm to generate triangulations of the exteriors of $2$-knots obtained by spinning $1$-knots. We give an implementation of the algorithm in \emph{Regina} and present triangulations of exteriors obtained from all $1$-knots with up to eight crossings.
	\end{abstract}
	
	\section{Introduction}
	
	One focus of geometric topology is the study of submanifolds of co-dimension $2$, since these admit knotted embeddings into their ambient manifolds. The most prominent instance of this is the case of embeddings of the circle into the $3$-sphere, better known under the name of {\em knot theory}. However, embeddings of the $2$-sphere into the $4$-sphere, known as {\em $2$-knots}, also admit a rich theory and provide valuable insights into the mysterious world of $4$-manifolds. These embeddings, or, more precisely, their complements, are the focus of this note.
	
	Specifically, given a $2$-knot $\twoknot: \mathbb{S}^2 \to \mathbb{S}^4$, we wish to describe {\em triangulations} of the complement $\mathbb{S}^4 \setminus \twoknot (\mathbb{S}^2)$ -- that is, a decomposition of this $4$-manifold into a (small) number of $4$-simplices (henceforth referred to as {\em pentachora}). 
	Triangulations and similar discretisations have had a tremendous impact on the study of $3$-manifolds, with notable successes, such as complexity-theoretic results for standard problems and the \emph{algorithmic} study of $3$-manifolds \cite{HLP,KuperbergAlgorithm,matveev2007algorithmic}. Moreover, censuses consisting of large numbers of triangulations are instrumental in modern geometric topology research, providing a source of examples, counterexamples, and ample opportunity to obtain intuition about important topological problems. We hope that analogous 4-dimensional results and censuses might prove to be equally useful. 
	
	While the latest iterations of censuses for ($1$-)knot complements, and closed $3$-manifolds contain billions of objects \cite{Burton11ISSAC,burton:LIPIcs.SoCG.2020.25,Thistlethwaite20CrossingKnots}, existing censuses in dimension four are much smaller \cite{BurkeBurtonSpreer2026-census}, and only a handful of sporadic examples of triangulations of $2$-knot complements are known, see e.g.\ \cite{budneyBurtonHillman-CappellShanesonComp,BurkeBurtonSpreer2026-census}. 
	At the same time, all smooth 4-manifolds admit essentially unique triangulations~\cite{cairns-1935,whitehead-1940}, and hence it is reasonable to expect that triangulations will be a useful tool in (smooth) 4-manifold topology as well. Furthermore, such triangulations will be useful in the construction of complicated triangulations of $4$-manifolds, or even the potential construction of $4$-manifolds homeomorphic, but not PL-isomorphic to well-known standard piecewise linear $4$-manifolds (so-called exotic piecewise linear structures on $4$-manifolds), as has already been outlined, for instance, in \cite{Burke-software} and \cite{Issa2017TriangulatingCH}.
	
	 \paragraph*{Our Contribution.} 
	The class of \emph{spun 2-knots} was first described by Artin~\cite{Artin25} (see also~\cite[Section 3J]{rolfsen-book}). One begins with a 1-knot $K$ in $\R^3$. Cutting it open yields a properly embedded knotted arc $K^\circ$ in the upper half-space~$\R^3_+$. Rotating $(\R^3_+, K^\circ)$ around the $xy$-plane produces a 2-knot in $\R^4$, denoted by $\spun(K)$. We present an algorithm, with implementation in \emph{Regina} \cite{regina}, to construct a triangulation of the complement of $\spun(K)$ in $S^4$ given a triangulation of the complement of $K^\circ$ in $S^3$.
	
	 \paragraph*{Outline.}  
	The procedure is relatively straightforward to describe, but implementing it faces intricate technical issues. These are described in detail in Section~\ref{sec:algo}. In Section~\ref{sec:examples} we then look at examples of small triangulations of spun knot complements. 
	
	 \paragraph*{Acknowledgements.} 
	This project began during a workshop at the MATRIX Institute in Creswick in 2025. We thank the institute for its hospitality, and the organisers and participants of the workshop for stimulating conversations. The first author is supported by the EPSRC grant EP/Y016270/1, ``Computable manifolds''. The fourth author is partially supported by the Australian Research Council under the Discovery Project scheme, project number DP220102588.
	
	\section{Background}
	
	In this section we briefly go over the most important notions used in this article, and provide references for further reading.
	
	The main focus of our work is on $2$-knots and their complements. A \emph{$2$-knot} is a locally flat piecewise linear embedding $\twoknot : \mathbb{S}^2 \to \mathbb{S}^4$ of the $2$-sphere into the $4$-sphere, considered up to ambient isotopy. 
	
	Given a $2$-knot $\twoknot : \mathbb{S}^2 \to \mathbb{S}^4$, we refer to $\mathbb{S}^4 \setminus \twoknot (\mathbb{S}^2)$ as the {\em complement of $\twoknot$}, and to $\mathbb{S}^4 \setminus U(\twoknot)$ as the \emph{exterior} of $\twoknot$. Here, $U(\twoknot)$ denotes an open tubular neighbourhood of $\twoknot (\mathbb{S}^2)$ in $\mathbb{S}^4$. By construction, the boundary of a $2$-knot exterior is homeomorphic to $\mathbb{S}^2 \times \mathbb{S}^1$. Exteriors of $2$-knots form an interesting family of piecewise linear $4$-manifolds with boundary. 
	For more information on $2$-knots and 4-manifolds, we refer the reader to, for instance,  \cite{hillman19892} and \cite{gompf19994} respectively.
	
	Since this article focuses on {\em triangulations} of complements of $2$-knots, we now introduce some notions about them.
	
	Given points $\{ v_0, \ldots , v_d \} \subset \mathbb{R}^d$ in general position, $\Delta^d = \operatorname{conv} (v_0, \ldots , v_d)$ is called a {\em $d$-simplex} with {\em vertices} $v_i$, $0 \leq i \leq d$. Given a $d$-simplex $\Delta^d$, every subset of its vertices spans a simplex of lower dimension, called a {\em face} of $\Delta^d$. Faces of dimension $0$ and $1$ are called {\em vertices} and {\em edges}, and faces of co-dimension $1$ are called {\em facets}, respectively. Moreover, we refer to simplices of dimensions $2$, $3$, and $4$ as triangles, tetrahedra, and {\em pentachora} respectively.
	
	A \emph{(generalised) $d$-dimensional triangulation} $\tri$ is a finite collection of $n$ abstract $d$-simplices, some or all of whose facets are affinely identified (``glued'') in pairs. 
	The \emph{real boundary} of $\tri$ consists of all the facets that are not identified with any other facets. We allow facets of the same $d$-simplex to be identified. The gluings defining $\tri$ also have the effect of merging faces of the simplices into equivalence classes, which we refer to as the \emph{faces} of $\tri$. We say that $\tri$ is {\em valid}, if the facet gluings never identify a face with itself along a non-identity map.
	
	The \emph{link} $\lk(v)$ of a vertex $v$ of $\tri$ is the ``frontier'' of a small regular neighbourhood of $v$. We treat vertex links as triangulated $(d-1)$-dimensional spaces, formed by inserting a small $(d-1)$-simplex into each corner of each $d$-simplex, and then joining together the $(d-1)$-simplices from adjacent $d$-simplices along their facets. This mirrors the traditional concept of a link in a simplicial complex, but is modified to support generalised triangulations. 
	
	The following notions are defined recursively over decreasing dimension, where a triangulation of a (piecewise linear) $0$-dimensional sphere is given by the disjoint union of two vertices. For simplicity, in the following we only consider valid generalised triangulations where all vertex links are triangulations of (piecewise linear) $(d-1)$-manifolds. Let $\tri$ be such a triangulation. We refer to the vertices of $\tri$ with links that are spheres as {\em finite} vertices, and to the vertices with non-sphere links as {\em ideal} vertices. If $\tri$ has at least one ideal vertex, we call $\tri$ an {\em ideal triangulation}. We can think of an ideal vertex as not being part of $\tri$, but instead representing a boundary component of $\tri$ of topological type the vertex link. We say that such boundary components form the {\em ideal} boundary of $\tri$. If all vertex links of $\tri$ are piecewise linear triangulations of spheres, then the triangulation $\tri$ is said to be {\em piecewise linear} or {\em closed}. If all vertex links of $\tri$ are piecewise linear triangulated spheres or balls, $\tri$ is a piecewise linear triangulation of a manifold with real boundary, where, as explained above, the real boundary is given by the unglued facets in the triangulation. 
	
	If $\tri$ is a piecewise linear triangulation, with or without real boundary, then the underlying set, or the {\em carrier}, of $\tri$ is a piecewise linear manifold denoted by $|\tri|$. For ideal triangulations $\tri$, the carrier denotes the underlying set of $\tri$ with ideal vertices removed. For more background on triangulations, see \cite{BurkeBurtonSpreer2026-census,Issa2017TriangulatingCH,KimYamadaCS4Spheres}.
	
	\section{The Algorithm}
	\label{sec:algo}
	
	Let $K : \mathbb{S}^1 \to \mathbb{S}^3$ be a (classical) $1$-knot in the $3$-sphere. We describe how to construct a triangulation $\tri$ of the complement of the $2$-knot $\spun (K)$. In an intermediate step, this triangulation will be a $4$-dimensional triangulation with real boundary, that is, a triangulation containing a set of boundary tetrahedra triangulating $\mathbb{S}^2 \times \mathbb{S}^1$, and hence a triangulation of the exterior of $\spun(K)$ in $\mathbb{S}^4$. This is the procedure we describe in this section. However, the final output is an ideal triangulation, with the unique ideal vertex representing $\spun(K)$, and hence a triangulation of the complement of $\spun(K)$ in $\mathbb{S}^4$. This is purely for efficiency~reasons: in practice, ideal triangulations are typically smaller and simplify more effectively than triangulations with triangulated real boundary.
	
	\begin{theorem}
		\label{thm:main}
		Given a long knot $K^\circ\subset B^3$, the procedure described below constructs a triangulation of the exterior of the spun knot $\operatorname{spun}(K)\subset S^4$.
	\end{theorem}
	
	Given a diagram of $K$, we cut $K$ open at a specified arc of the diagram. We then use a variant of Weeks' algorithm (\cite{menasco2005handbook}, Chapter 10, Section 3) to obtain an ideal triangulation of the complement of the cut-open (long) knot in a triangulation of the $3$-ball. The $2$-sphere bounding the $3$-ball is represented using four boundary triangles with two points pinched together at the ideal vertex $v$ representing the cut-open knot. 
	Denote a simplified version of this triangulation by $\tri'$. A schematic of $\tri'$ is depicted in Figure \ref{fig:longKnotInB3}.
	
	\begin{figure}[h!]
		\centering
		\includegraphics[width=0.55\linewidth]{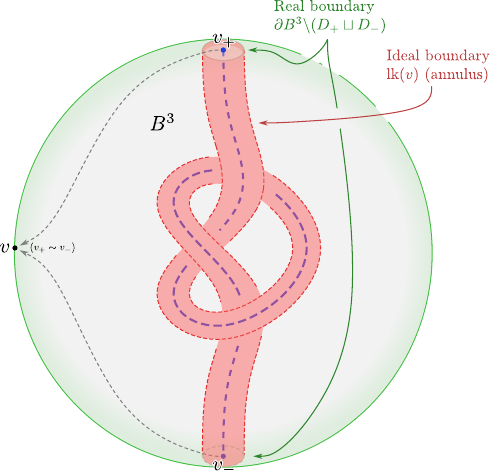}
		\caption{Schematic illustration of $\tri'$ (note that the endpoints, $v_\pm$, of the long knot are identified as part of the single ideal vertex $v$)}
		\label{fig:longKnotInB3}
	\end{figure}
	
	It follows that $\tri'$ has a mixture of real and ideal boundary, where the real part of the boundary (the four boundary triangles) describes the sphere with the interiors of two discs, $D_+$ and $D_-$, removed, and the ideal part of the boundary (the link of $v$) describes the annulus bounding the long knot inside this ball. In particular, in the language of {\em Regina}, $\tri'$ is invalid. However, truncating $v$ in $\tri'$ results in a valid triangulation of the knot exterior with real boundary. The triangulation $\tri'$ (simplified or not) is the output of the {\em Regina} function \texttt{D.longComplement()}, where $D$ is the input diagram of $K$. Note that the output of Weeks' algorithm can be obtained from the non-simplified output of \texttt{D.longComplement()} by simply gluing together the triangles that meet the specified cut-arc.
	
	With our triangulation $\tri'$ at hand, we realise Artin's spinning construction \cite{Artin25} as follows: We start by truncating the vertex $v$ to obtain a triangulation $\operatorname{trunc} (v,\tri')$. This produces additional boundary triangles triangulating an annulus, and overall a triangulation of the exterior of $K$ with real torus boundary. 
	We refer to the part of the boundary coming from the truncation annulus (the ideal boundary of $\tri'$) as the ``red'' part, and to the rest of the boundary (the real boundary of $\tri'$) as the ``green'' part of the boundary. 
	In the following steps of the construction we will leave the red part of the boundary, and subcomplexes arising from it, unchanged or {\em locked}.
	
	\begin{figure}[htb]
		\centering
		\includegraphics[width=0.4\linewidth]{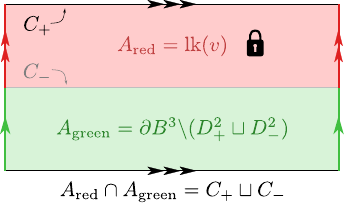}
		\caption{Structure of $\partial|\operatorname{trunc}(v,\tri')|$}
		\label{fig:truncBdryTorus}
	\end{figure}
	
	Next, we build an $\mathbb{S}^1$-bundle over $\operatorname{trunc} (v,\tri')$. For each tetrahedron, we construct a tetrahedral prism $\Delta^3 \times [0,1]$, containing $82$ pentachora: Every edge of type $\texttt{vertex}\times [0,1]$ is subdivided once. As a result, every rectangular wall of type $\texttt{edge}\times [0,1]$ is coned over its centre producing a subdivision into six triangles. The $3$-dimensional prism walls of type $\texttt{triangle}\times [0,1]$ contain $3 \cdot 6 + 2 = 20$ triangles that are coned over the centre of the prism, as illustrated in Figure~\ref{fig:prism}. Coning over the centre of the $4$-dimensional prism produces $4 \cdot 20 + 2 = 82$ pentachora overall. We use this particular subdivision because it admits an involution exchanging $\Delta^3 \times \{x\}$ with $\Delta^3 \times \{1-x\}$, which is required in the later ``folding'' step of the construction. 
	
	\begin{figure}[htb]
		\begin{center}
			\includegraphics[height=7cm]{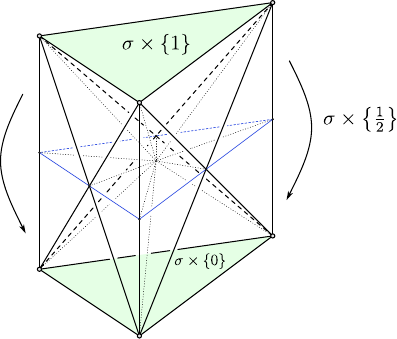}
		\end{center}
		\caption{Subdivision of a $3$-dimensional prism wall of $\Delta^3 \times [0,1]$ into $20$ tetrahedra, with an indication of how it is folded in half
			\label{fig:prism}}
	\end{figure}
	
	We then glue these prisms together in a manner that follows the gluings of the original tetrahedra. Next, we take the two copies of the triangulation that are formed from the tetrahedra at the two ends of these prisms, and glue these together according to the identity gluing to obtain a triangulation $\tri''$ of an $\mathbb{S}^1$-bundle over $\operatorname{trunc} (v,\tri')$.  
	
	By construction, $\tri''$ is a $4$-dimensional triangulation with real boundary a $3$-torus. This $3$-torus boundary splits into a red and a green part coming from the splitting of the boundary of $\operatorname{trunc} (v,\tri')$. Naturally, each part is homeomorphic to a thickened torus, i.e.\ a trivial $\mathbb{S}^1$-bundle over an annulus.
	
	\begin{figure}[htb]
		\centering
		\includegraphics[width=0.5\linewidth]{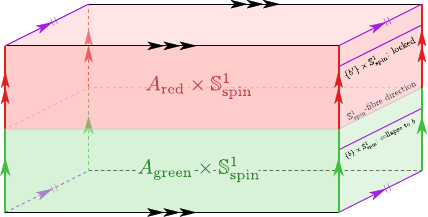}
		\caption{Structure of $\partial\tri''\cong T^3$. Fibres over the green part will be collapsed; fibres over the red part are locked.}
		\label{fig:threeTorusBdry}
	\end{figure}
	
	For each point $b$ on the green part of $\partial \,|\operatorname{trunc} (v,\tri')|$, we collapse the fibre $b \times \mathbb{S}^1$ in $|\tri''|$ to a single point. This can be equivalently described as attaching a copy of ``the green part of the boundary of $\operatorname{trunc} (v,\tri')$'' $\times \,\, \mathbb{D}^2$ to $|\tri''|$ so that, for each point $b$ on the green part of the boundary, the fibre $b \times \mathbb{S}^1$ becomes the boundary of the corresponding disc $b \times \mathbb{D}^2$. We implement this as follows: for each triangle of the green part of the boundary of $\operatorname{trunc} (v,\tri')$, we fold the corresponding prism in $\tri''$ in half, so that the top half folds onto the bottom. The red part of the boundary, on the other hand, is locked and will not be folded.
	
	\begin{proof}[Proof of Theorem~\ref{thm:main}]
		The proof is a review of the above combinatorial operations, verifying that these yield a triangulation of the exterior of $\spun(K)$. The reader may find it helpful to refer to Figures~\ref{fig:longKnotInB3}, \ref{fig:truncBdryTorus} and \ref{fig:threeTorusBdry} along the way.
		
		Let \(a\subset B^3\) denote the cut-open knot arc, with endpoints
		\(v_+\) and \(v_-\) on \(\partial B^3\). Choose a regular neighbourhood
		\(N(a)\cong [0,1]\times \mathbb D^2\) whose intersection with
		\(\partial B^3\) is the union of two discs
		\[
		D_+ \sqcup D_- \subset \partial B^3,
		\]
		where \(D_\pm\) is a small disc about \(v_\pm\). Then
		\[
		X' = B^3\setminus \operatorname{int}N(a)
		\]
		is the exterior of the long knot. Its boundary decomposes as
		\[
		\partial X' = A_{\mathrm{green}}\cup A_{\mathrm{red}},
		\]
		where
		\[
		A_{\mathrm{green}}
		=
		\partial B^3\setminus \operatorname{int}(D_+\sqcup D_-)
		\]
		is the part inherited from the boundary of the ball, and
		\[
		A_{\mathrm{red}}
		=
		\partial N(a)\setminus \operatorname{int}(D_+\sqcup D_-)
		\]
		is the annular boundary coming from the regular neighbourhood of the long
		knot. These two annuli meet along the two circles
		\[
		C_+ \sqcup C_- =
		\partial D_+ \sqcup \partial D_-.
		\]
		All this is illustrated in Figure~\ref{fig:longKnotInB3}. The triangulation \(\operatorname{trunc}(v,\tri')\) is a triangulation of
		this space \(X'\), with the red and green parts of the boundary corresponding
		to the two annuli above, see Figure~\ref{fig:truncBdryTorus}.
		
		Recall the standard description of \(S^4\) as the result of spinning a
		\(3\)-ball:
		\[
		S^4
		\cong
		(B^3\times \mathbb S^1_{\mathrm{spin}})
		\cup_{\partial B^3\times \mathbb S^1_{\mathrm{spin}}}
		(\partial B^3\times \mathbb D^2_{\mathrm{spin}}).
		\]
		Here the first piece records the \(3\)-ball as it rotates through the ``spin circle'', and the second piece fills in the fibres over the boundary
		\(\partial B^3\). Under this decomposition, the Artin spin of the long knot
		arc \(a\) is given by
		\[
		\Sigma
		=
		(a\times \mathbb S^1_{\mathrm{spin}})
		\cup
		(v_+\times \mathbb D^2_{\mathrm{spin}})
		\cup
		(v_-\times \mathbb D^2_{\mathrm{spin}}).
		\]
		The first piece \(a\times \mathbb S^1_{\mathrm{spin}}\) is an annulus, and
		the two discs \(v_+\times \mathbb D^2_{\mathrm{spin}}\) and
		\(v_-\times \mathbb D^2_{\mathrm{spin}}\) cap its two boundary components.
		Thus \(\Sigma\) is a \(2\)-sphere. Moreover, because the arc \(a\) was
		obtained by cutting the original knot \(K\) open, this \(2\)-sphere is
		precisely the Artin spin \(\operatorname{spun}(K)\).
		
		Now take the exterior of this sphere in the above decomposition of \(S^4\).
		Inside \(B^3\times \mathbb S^1_{\mathrm{spin}}\), removing a regular
		neighbourhood of \(a\times \mathbb S^1_{\mathrm{spin}}\) leaves
		\[
		X'\times \mathbb S^1_{\mathrm{spin}}.
		\]
		Inside the filling piece
		\(\partial B^3\times \mathbb D^2_{\mathrm{spin}}\), removing the two disc
		neighbourhoods
		\[
		D_+\times \mathbb D^2_{\mathrm{spin}}
		\quad\text{and}\quad
		D_-\times \mathbb D^2_{\mathrm{spin}}
		\]
		leaves
		\[
		\bigl(\partial B^3\setminus \operatorname{int}(D_+\sqcup D_-)\bigr)
		\times \mathbb D^2_{\mathrm{spin}}
		=
		A_{\mathrm{green}}\times \mathbb D^2_{\mathrm{spin}}.
		\]
		Therefore the exterior of the spun sphere $\Sigma$ is
		\[
		S^4\setminus \operatorname{int}N(\operatorname{spun}(K))
		\cong
		(X'\times \mathbb S^1_{\mathrm{spin}})
		\cup_{A_{\mathrm{green}}\times \mathbb S^1_{\mathrm{spin}}}
		(A_{\mathrm{green}}\times \mathbb D^2_{\mathrm{spin}}).
		\]
		
		This is exactly the topological operation performed in the algorithm.
		The triangulation \(\tri''\) triangulates
		\(X'\times \mathbb S^1_{\mathrm{spin}}\). For each point
		\(b\in A_{\mathrm{green}}\), collapsing the fibre
		\[
		\{b\}\times \mathbb S^1_{\mathrm{spin}}
		\]
		is equivalent to attaching the disc
		\[
		\{b\}\times \mathbb D^2_{\mathrm{spin}},
		\]
		see Figure~\ref{fig:threeTorusBdry}. Thus the green fibre collapse attaches
		\(A_{\mathrm{green}}\times \mathbb D^2_{\mathrm{spin}}\), exactly as in the
		exterior of the spun sphere. The red annulus
		\(A_{\mathrm{red}}\) comes from the boundary of the regular neighbourhood of
		the long knot, and so its product
		\(A_{\mathrm{red}}\times \mathbb S^1_{\mathrm{spin}}\) is part of the
		boundary of the tubular neighbourhood of the spun \(2\)-sphere. 
	\end{proof}
	
	The procedure is implemented in {\em Regina} \cite{regina} as the function \texttt{boundarySpin()} and, in particular, the $4$-dimensional example \texttt{spun()}, which takes as input a knot, and optionally a specific strand of the knot at which to perform the cut. The implemented procedure 
	additionally simplifies the triangulation, makes it ideal, and simplifies it again before returning it. Due to the typically large size of the triangulations produced by the algorithm `out-of-the-box' the \texttt{spun()} function is currently only accessible via the C++/Python APIs of \emph{Regina} (e.g.\ via \texttt{regina-python} or the Python shell accessible in the GUI). An example, in Python, using the trefoil as input, is executed as follows.
	
	\begin{lstlisting}
>>> tref = ExampleLink.trefoil()
>>> spunTref = Example4.spun(tref)
>>> spunTref
<regina.Triangulation4: Ideal orientable 4-D triangulation, f = ( 2 24 86 105 42 )>
\end{lstlisting}
	
	\begin{remark}
		The construction of an ideal triangulation of the $\spun(K)$ complement relies on the boundary of $\tri''$ to not be simplified too much. This property is satisfied as long as the red part of the boundary is a triangulation homeomorphic to an annulus with the two components of its boundary disjoint. This is ensured by locking the entire truncation surface of $\operatorname{trunc} (v,\tri')$.
		
		Also note that working with ideal triangulations throughout this process is difficult: if $K$ is represented as a single point, spinning $K$ on the level of ideal triangulations will produce an edge representing the $2$-knot, which, as a triangulation, is more difficult to handle.
	\end{remark}
	
	\begin{remark}
		The construction is tailored specifically to Artin spinning.
		It is natural to ask whether analogous constructions can be
		developed for twist-spinning or more general deform-spinning
		operations. Such extensions constitute avenues for future work.
	\end{remark}

	\subsection{Running times and triangulation sizes}		
	Running times for low-crossing knots of our implementation are listed in Table~\ref{tab:examples}. On average, they are $6$, $45$, $50.3$, $114.33$, $292.71$ and $626.04$ seconds for $3$-, $4$-, $5$-, $6$-, $7$-, and $8$-crossing knots. For knots with $\geq 5$ crossings this roughly amounts to a doubling of running time per crossing (i.e.\ an experimentally observed running time of roughly $O(c\cdot 2^c)$ for $c$-crossing knots). The bulk of this running time is used for simplifying large intermediate triangulations. For this reason, it pays to analyse running times from the viewpoint of sizes of intermediate triangulations that our methods must handle: 
	
	Starting from a knot diagram with, say, $c$ crossings, the triangulation of the long knot complement is of size roughly $4c$ before simplification. Truncating the ideal vertex inflates this size by a similar or moderately larger linear constant (largely depending on the triangulation). It follows that the truncated 3-dimensional triangulation has size roughly one order of magnitude higher than $c$. Passing to the bundle construction inflates this by a multiplicative factor of $82$. Hence, altogether, if we construct the spun knot of a $c$-crossing diagram, we have to expect to go through a $4$-dimensional triangulation with around $1000 c$ pentachora.
	
	While this seems like bad news, we can observe that the intermediate triangulation size only grows linearly in the crossing number of the diagram, and this linear upper bound is guaranteed by the construction (Weeks' algorithm, truncation, bundle construction), even without intermediate simplification.
	
	Comparing this to the actual output size of the triangulations presented in Table~\ref{tab:examples}, we see that the average sizes of the output triangulations are $18.11$, $17.06$, and $17.31$ times the crossing number for $6$-, $7$-, and $8$-crossing knots, while the sizes of our triangulations after simplification are $4.22$, $3.8$, and $4.56$ times the crossing number -- arguably a fairly small multiplier.
	
	\begin{table}[htbp]
		\centering
		\begin{tabular}{@{}lcrr@{}}
			\toprule
			$K$ & min.\ $f_4$ of $\mathbb{S}^4 \setminus \spun(K)$ & initial size & time$^{\ast}$ \\
			\midrule
			$3a_1$      &  6 &  30 &    6 \\
			\addlinespace
			$4a_1$      & 16 &  58 &   45 \\
			\addlinespace
			$5a_1$      & 18 &  72 &   72 \\
			$5a_2$      & 10 &  52 &   29 \\
			\addlinespace
			$6a_3$      & 22 &  94 &  112 \\
			$6a_2$      & 26 & 102 &   99 \\
			$6a_1$      & 28 & 112 &  132 \\
			\addlinespace
			$7a_4$      & 22 & 104 &  138 \\
			$7a_5$      & 26 &  86 &   89 \\
			$7a_6$      & 28 & 120 &  160 \\
			$7a_3$      & 28 & 140 &  563 \\
			$7a_2$      & 32 & 126 &  921 \\
			$7a_1$      & 38 & 178 &  141 \\
			$7a_7$      & 12 &  82 &   37 \\
			\addlinespace
			$8a_{11}$   & 26 &  80 &   99 \\
			$8a_8$      & 28 &  80 &   72 \\
			$8a_{18}$   & 34 & 124 &  192 \\
			$8a_{17}$   & 28 & 120 &  422 \\
			$8a_{13}$   & 48 & 148 &  146 \\
			$8a_6$      & 32 & 134 &  211 \\
			$8a_{10}$   & 34 & 124 &  160 \\
			$8a_{16}$   & 34 & 120 &  382 \\
			$8a_4$      & 40 & 148 & 1108 \\
			$8a_9$      & 40 & 134 & 1226 \\
			$8a_7$      & 40 & 162 & 1625 \\
			$8a_3$      & 42 & 158 &  548 \\
			$8a_5$      & 42 & 178 & 1122 \\
			$8a_1$      & 42 & 170 & 1240 \\
			$8a_2$      & 46 & 170 &  552 \\
			$8a_{15}$   & 44 & 178 & 1028 \\
			$8a_{14}$   & 50 & 190 & 1653 \\
			$8a_{12}$   & 48 & 220 &  862 \\
			$8n_1$      & 24 & 104 &  282 \\
			$8n_2$      & 32 & 116 &  189 \\
			$8n_3$      & 12 &  50 &   28 \\
			\bottomrule
		\end{tabular}
		\caption{Smallest known values of $f_4$ of spun-knot complements of low-crossing knots. Isomorphism signatures of the corresponding triangulations are given in Appendix~\ref{sec:app}. $^{\ast}$ Running time in seconds for calling the spinning construction as implemented in {\em Regina}. The presented triangulations have been obtained from the initial {\em Regina} output by extensive simplification attempts.}
		\label{tab:examples}
	\end{table}
	
	\section{Small triangulations of \texorpdfstring{$2$}{2}-knot complements}
	\label{sec:examples}
	
	We apply the algorithm described in Section~\ref{sec:algo} to the census of 1-knots up to $8$ crossings. The results are summarised in Table~\ref{tab:examples}. Isomorphism signatures -- compact strings encoding the triangulation uniquely up to combinatorial isomorphism \cite{Burton-IsoSigs} -- of the triangulations are given in Appendix \ref{sec:app}.

	Since the intermediate triangulations produced by the construction are quite large, it is non-trivial to extend this to a larger census of knots. The bottleneck for this is not the construction itself, but the simplification of the resulting output triangulations; see the last paragraph of Section~\ref{sec:algo}.
	
	\begin{remark}
		In Table \ref{tab:examples}, all of the triangulations have a single vertex, and by the Dehn--Sommerville equations and the equation for the Euler characteristic, this means that the $f$-vector is entirely determined by the number of edges $f_1$ -- in particular the number of pentachora $f_4$ is equal to $2f_1$.
	\end{remark}
	
	Among all spun knots arising from knots with at most eight crossings, the smallest triangulation found is the spun trefoil with just six pentachora. Using the ideal census with up to six pentachora, we are able to verify that six pentachora is indeed minimal for the spun trefoil. With the exception of the spun trefoil, we do not know whether any of the other triangulations obtained are minimal.
	
	The observant reader might notice that the complements of spun torus knots admit much smaller triangulations than the other knots -- this will be the topic of follow-up work in which we will investigate infinite families of such triangulations.

	\clearpage
	
	\bibliographystyle{plainurl}
	\bibliography{bibliography}

\begin{thebibliography}{10}

\bibitem{Artin25}
Emil Artin.
\newblock Zur {I}sotopie zweidimensionaler {F}l\"achen im {$R_4$}.
\newblock {\em Abh. Math. Sem. Univ. Hamburg}, 4(1):174--177, 1925.
\newblock \href {https://doi.org/10.1007/BF02950724}
  {\path{doi:10.1007/BF02950724}}.

\bibitem{budneyBurtonHillman-CappellShanesonComp}
R.~Budney, B.~A. Burton, and J.~Hillman.
\newblock Triangulating a {C}appell-{S}haneson knot complement.
\newblock {\em Math. Res. Lett.}, 19(5):1117--1126, 2012.
\newblock \href {https://doi.org/10.4310/MRL.2012.v19.n5.a12}
  {\path{doi:10.4310/MRL.2012.v19.n5.a12}}.

\bibitem{BurkeBurtonSpreer2026-census}
Rhuaidi Burke, Benjamin Burton, and Jonathan Spreer.
\newblock Small triangulations of 4-manifolds and the 4-manifold census.
\newblock {\em Discrete {\&} Computational Geometry}, Feb 2026.
\newblock \href {https://doi.org/10.1007/s00454-026-00818-w}
  {\path{doi:10.1007/s00454-026-00818-w}}.

\bibitem{Burke-software}
Rhuaidi~Antonio Burke.
\newblock Practical software for triangulating and simplifying 4-manifolds.
\newblock {\em Journal of Computational Geometry}, 16, 2025.
\newblock \href {https://doi.org/10.20382/JOCG.V16I2A4}
  {\path{doi:10.20382/JOCG.V16I2A4}}.

\bibitem{Burton11ISSAC}
Benjamin~A. Burton.
\newblock Detecting genus in vertex links for the fast enumeration of
  3-manifold triangulations.
\newblock In {\em Proceedings of the 36th international symposium on Symbolic
  and algebraic computation (ISSAC 2011)}, pages 59--66. ACM, 2011.
\newblock \href {https://doi.org/10.1145/1993886.1993901}
  {\path{doi:10.1145/1993886.1993901}}.

\bibitem{Burton-IsoSigs}
Benjamin~A. Burton.
\newblock The {P}achner graph and the simplification of 3-sphere
  triangulations.
\newblock In {\em 27th International Symposium on Computational Geometry
  (SoCG)}, pages 153--162. ACM, New York, 2011.
\newblock \href {https://doi.org/10.1145/1998196.1998220}
  {\path{doi:10.1145/1998196.1998220}}.

\bibitem{burton:LIPIcs.SoCG.2020.25}
Benjamin~A. Burton.
\newblock {The Next 350 Million Knots}.
\newblock In Sergio Cabello and Danny~Z. Chen, editors, {\em 36th International
  Symposium on Computational Geometry (SoCG 2020)}, volume 164 of {\em Leibniz
  International Proceedings in Informatics (LIPIcs)}, pages 25:1--25:17,
  Dagstuhl, Germany, 2020. Schloss Dagstuhl -- Leibniz-Zentrum f{\"u}r
  Informatik.
\newblock URL:
  \url{https://drops.dagstuhl.de/entities/document/10.4230/LIPIcs.SoCG.2020.25},
  \href {https://doi.org/10.4230/LIPIcs.SoCG.2020.25}
  {\path{doi:10.4230/LIPIcs.SoCG.2020.25}}.

\bibitem{regina}
Benjamin~A. Burton, Ryan Budney, William Pettersson, et~al.
\newblock Regina: Software for low-dimensional topology.
\newblock {\tt http://\allowbreak regina-normal.\allowbreak github.\allowbreak
  io/}, 1999--2025.

\bibitem{cairns-1935}
S.~S. Cairns.
\newblock Triangulation of the manifold of class one.
\newblock {\em Bull. Am. Math. Soc.}, 41:549--552, 1935.
\newblock \href {https://doi.org/10.1090/S0002-9904-1935-06140-3}
  {\path{doi:10.1090/S0002-9904-1935-06140-3}}.

\bibitem{gompf19994}
R.E. Gompf and A.~Stipsicz.
\newblock {\em 4-Manifolds and Kirby Calculus}.
\newblock Graduate studies in mathematics. American Mathematical Society, 1999.
\newblock URL: \url{https://books.google.com.au/books?id=ahLKzRUTBbUC}.

\bibitem{HLP}
Joel Hass, Jeffrey~C. Lagarias, and Nicholas Pippenger.
\newblock The computational complexity of knot and link problems.
\newblock {\em J. ACM}, 46(2):185--211, 1999.
\newblock \href {https://doi.org/10.1145/301970.301971}
  {\path{doi:10.1145/301970.301971}}.

\bibitem{hillman19892}
J.~Hillman.
\newblock {\em 2-Knots and Their Groups}.
\newblock Australian Mathematical Society Lecture Series. Cambridge University
  Press, 1989.
\newblock URL: \url{https://books.google.com.au/books?id=ETo4AAAAIAAJ}.

\bibitem{Issa2017TriangulatingCH}
Ahmad Issa.
\newblock Triangulating {C}appell-{S}haneson homotopy 4-spheres.
\newblock Master's thesis, University of Melbourne, 2017.
\newblock URL: \url{https://api.semanticscholar.org/CorpusID:125974898}.

\bibitem{KimYamadaCS4Spheres}
Min~Hoon Kim and Shohei Yamada.
\newblock Ideal classes and {C}appell-{S}haneson homotopy 4-spheres.
\newblock {\em Kyungpook Mathematical Journal}, 63(3), 2023.
\newblock \href {https://doi.org/10.5666/KMJ.2023.63.3.373}
  {\path{doi:10.5666/KMJ.2023.63.3.373}}.

\bibitem{KuperbergAlgorithm}
Greg Kuperberg.
\newblock {Algorithmic Homeomorphism of 3-Manifolds as a Corollary of
  Geometrization}.
\newblock {\em Pac. J. Math.}, 301(1), 2019.
\newblock \href {https://doi.org/10.2140/pjm.2019.301.189}
  {\path{doi:10.2140/pjm.2019.301.189}}.

\bibitem{matveev2007algorithmic}
S.~Matveev.
\newblock {\em Algorithmic Topology and Classification of 3-Manifolds}.
\newblock Algorithms and Computation in Mathematics. Springer Berlin
  Heidelberg, 2007.
\newblock URL: \url{https://books.google.com.au/books?id=vFLgAyeVSqAC}.

\bibitem{menasco2005handbook}
William Menasco and Morwen Thistlethwaite, editors.
\newblock {\em Handbook of knot theory}.
\newblock Elsevier B. V., Amsterdam, 2005.

\bibitem{rolfsen-book}
Dale Rolfsen.
\newblock {\em Knots and links}, volume No. 7 of {\em Mathematics Lecture
  Series}.
\newblock Publish or Perish, Inc., Berkeley, CA, 1976.

\bibitem{Thistlethwaite20CrossingKnots}
Morwen~B Thistlethwaite.
\newblock The enumeration and classification of prime 20-crossing knots.
\newblock {\em Algebr. Geom. Topol.}, 25:329--344, 2025.
\newblock \href {https://doi.org/10.2140/agt.2025.25.329}
  {\path{doi:10.2140/agt.2025.25.329}}.

\bibitem{whitehead-1940}
J.~H.~C. Whitehead.
\newblock On {{\(C^1\)}}-complexes.
\newblock {\em Ann. Math. (2)}, 41:809--824, 1940.
\newblock \href {https://doi.org/10.2307/1968861} {\path{doi:10.2307/1968861}}.

\end{thebibliography}
		
	\appendix
	
	\section{Isomorphism signatures}
	\label{sec:app}
	
	In this appendix we list isomorphism signatures of the smallest triangulations of spun-knot complements that we found. 
	
	To reconstruct the triangulations in {\em Regina}---
	\begin{itemize}
		\item via the GUI: produce a new 4-manifold triangulation by selecting type of triangulation ``From isomorphism signature'' and pasting in the string \texttt{isoSig}; or
		\item via the Python API: \texttt{T = Triangulation4.fromIsoSig(isoSig)}
	\end{itemize}
	
	Knot complements of $\spun (K)$ are listed in the order in which they appear in the {\em Regina} census data files (which list hyperbolic alternating knots first, hyperbolic non-alternating second, and torus knots third). Notation follows Dowker--Thistlethwaite.
	
	\subsection*{Complement of the spun trefoil $\mathbb{S}^4 \setminus \spun (3a_1)$}
		
	\is{gLMMMQbcddeeffff0baaYaYaeawbpapaYaYa}	
	
	\subsection*{Complement of the spun figure-eight knot $\mathbb{S}^4 \setminus \spun (4a_1)$}
		
	\is{qLLLLMLvAQQQQQcbccfehikhplmpmlklonmopnpoWaYa2aEaMbaaqbRaYaZaxbkbobkb6a5aMbKbebDaKbgayababa}	
	
	\subsection*{$\mathbb{S}^4 \setminus \spun (5a_1)$}
		
	\is{sLLLLvvQMQQAAQQQcbgflnohgonpjnokpkpmonrqpqrraagarbaaaaxbxbPbga9aBaKb1bvafaJazbnaQaYaYaKaZa0baaoasafa}
	
	\subsection*{$\mathbb{S}^4 \setminus \spun (5a_2)$}
	
	\is{kLLLMQLQQccecegfgfiijjjijiZbXbXaFaaaQa0bQa1a1a1a1abbvaFava}
	
	\subsection*{$\mathbb{S}^4 \setminus \spun (6a_3)$}
		
	\is{wLLLLvvwAMPQQMQMQQQccccgfinjpnqrrkmrlosoptsvtupqrttvuvLaybIbTaPb2aVambEbtbqalaaaHb1bdb1bbaEblb4apaaa3bua0bPbPbcaoaoa2ababa}
	
	\subsection*{$\mathbb{S}^4 \setminus \spun (6a_2)$}
		
	\is{ALLLLAwMLwLAwPMQPQQMQQkabcgdgiimonmjskswntwututqyutyvuzxxvzxyzzfafafaPafayawawaearbPaYbvaCagaqblaOaobzbeaTaEa4afaha7aIahamaxarbDbqafaFaGa+aibOa}

\subsection*{$\mathbb{S}^4 \setminus \spun (6a_1)$}

\is{CLLLLvPwPvQMAPMzQAPPQQQQcbccglkjhloqsipnkuvvqvtwryxAAswtywuxAxyzzBBBWa1a2aabaaaaMbRaRaQaabKbbaabWa2apaaaaaEaGaMbGawakbObaalaTbTbOa0aqb2aDbEaDbFbYayawawaPb}

\subsection*{$\mathbb{S}^4 \setminus \spun (7a_4)$}

\is{wLLLLMLLQvAQPAzQQQQcabcefhihjjmkkqpnlqnorrvuutqssvvtuvcacacaDbDbGaPbGaba6aEbRb3bkbGavagaXbMblbQbFaGbGbAapacaObWabbGa2ambfa}

\subsection*{$\mathbb{S}^4 \setminus \spun (7a_5)$}

\is{ALLLLAMvzzMPzQMQwQQAQQkbeegdihijmmrrqlnsqnsporvuuvxtuxxzvxywyzz4a+a+aGb2aaaAaga+aMaMa-a-aAaobpbLbCaQbgafavapbXbQb+aFaeavafbVbVbEa2ambgbcaqb+a2a}

\subsection*{$\mathbb{S}^4 \setminus \spun (7a_6)$}

\is{CLLLLvvzQzPMQQvAPPQAQQQQcabchnqqgrmrjukpsmuqvowrztAAAutvuBzxyxzAzByBfafafaaaaaaaMafaaada7aHbtaHbIa7a0aEaqbaaQbdaQb7aybPaPaPafbRakb2azbqbQbqbtb5aca5aRafa2a}

\subsection*{$\mathbb{S}^4 \setminus \spun (7a_3)$}

\is{CLLLLLAzPvQALzQAAPPQMQQQcbcchgikmkinjmkqmtvuqwwpprwtAzsvuxBxyBzyBAABWa1a2aaaaaEaNbaaJafaQaRaaaRa7aEaaaya2awaaaaabayaTbWaTbJb7afatbQbzbaaWbfagawayaja2aTbIb}

\subsection*{$\mathbb{S}^4 \setminus \spun (7a_2)$}

\is{GLLAvzLwALQPPQvzPPALAQQQPQQkcbdcgihlklqrklkttpooqpzyvAwxCuyyCwDByDEzBAFECEEFFEagaqbHbWaiboaHbvaCawbtacbCacbhahaWajbjbmbyaaaMa3bhbmajbiabajbjbqb1a+agb1a+aQbyaMbqbUajasasajajbya}

\subsection*{$\mathbb{S}^4 \setminus \spun (7a_1)$}

\is{MLLLvLAvAPLLLwAwwQMQQPMQALQQQMQQkabfllelnmkrpnxwxorEEtCrCuxEztwGxCFHvyzKIKIHFIBJDDLHHJGKJLLfafalaaaqbfaqblaWboaaaebaaaaNbaaoaja7aaaTb0akaea+apbEboaQbYb3apbZbLaaabayaTbtaJbybyb6asayboaVayafaabgb+atbcaRafaUafa}

\subsection*{$\mathbb{S}^4 \setminus \spun (7a_7)$}

\is{mLLAMLvQQQQcdecdgglhlkiijlkkklAbybXaXaFaaaZbaaCbsbpa1b1bbaQaaazbzbZa}

\subsection*{$\mathbb{S}^4 \setminus \spun (8a_{11})$}

\is{ALLLLzPwLALMPQwPQMQAQQkcbgfjiikgokpnqqrnmpnuuwxswtssvtyzvxzzxzyaagarbaalaaaaaaagarbka8atataxa8aGbGbaaHbPbybEaGbbavbsaGa+aGbvatajbyafbmbcafa9afa}

\subsection*{$\mathbb{S}^4 \setminus \spun (8a_8)$}

\is{CLLLLvPwMzPAAQMPwPMPQQQQccbgfiijhgqnllkqlssqottpqurvzxuBzByAwyAzyBBAaagarbaaaaaa0bPbgala0bXbCaBblagbaaaaDbvbna2aXaJaPbga0brbZagaIa2anaBbkaoaBbXbkanayayaga}

\subsection*{$\mathbb{S}^4 \setminus \spun (8a_{18})$}

\is{ILLLLLLPwzLvPvwwQQQAQMQQQQQPQcabffkmhhlostsvtoBDpwqByuxACvCCGGyuxDyEGAAAEBzDHEFHGHcaca3baadaiaXbXbaaiaIaaanbKaGa1bdaiaibIbtbDbXaXaIbSbSb1bSbmaIaSbwavaqbGb2a4aIaWaWa0aKbKbfabaaacaoaiaFbga}

\subsection*{$\mathbb{S}^4 \setminus \spun (8a_{17})$}

\is{CLLLLvPAMwLQPQvMMQLQPQQQccbgfiijmhglklpnpqpnpouqxvvxtyzyBAwzxxzByABBaagarbaaaaaaGbFaPbgaJaHbJadbaaaaaadbQaIagaPbgaIazalarbIbaa0b0bHaAbIaCbTbCboatbgaTbcbcb}

\subsection*{$\mathbb{S}^4 \setminus \spun (8a_{13})$}

\is{WLLLvQLLwLvvwMvwMMzQQvMPwPQPQMQPQMQQQQQQQabgihjlmkqqooCzyEzystusrwvFFwDxHEIGDMALBOEIQQPUOJIHTPSNNQMMLUSRSPSVRTTVUVcacaaaaahaaaiaaaHaaahaYbYbaaHaPbaaaasacavaPbLbcafajaiaiasaiacaWbWbqaYbPbbbcafbgaKaibPbkagakaUagaBbbbqbOb5aKbtbLbfafawava0aLayanbqbnbVbMb-aPaYbqbqb}

\subsection*{$\mathbb{S}^4 \setminus \spun (8a_6)$}

\is{GLLLLLAPPvAAALAQzPAQwMPMQQQkbcchkjgmhkjpkjmsmvvpovwqsrxzuuAAywzBzDDCBCECFFFEFWa1a2aaa7aWaWaaa1afaJaaafbTbWa7aEaaaaavayadaaaTbTbPbaaDaba5aNbNbebya2aubTbdbjbYa5abaab6auaHbcayaba}

\subsection*{$\mathbb{S}^4 \setminus \spun (8a_{10})$}

\is{ILLLLLAzPwPPAvAAPwPwMMQQMQQQQcbcchgikmkiopmqppstvvvoxsCsAwuyzFGyzxGDGACHHBEGFFEHEHWa1a2aaaaaEaNbaaJafaaaaaaaaaWaWalaaaYaYaYayadaJaEafaJbtbQbobhbaaaaAaEbyaga1ajaFaJaWaWagaRaIbTb1afaVbPbaa}

\subsection*{$\mathbb{S}^4 \setminus \spun (8a_{16})$}

\is{ILLLLLLAwzMPAwPPPPwMAwQQQPQQQcbccfgkijmlooostpptwnustrxswtAAwzCFEGBDyzCHDDFEEGFGHHWaYa2aEaaaaaqbKaKavaLatata8a3aQaQaZaaaya1aOaYbKa1aKavavaQaQaKadaTaiaaaWawbaacayaybaa2aKacaMbMb+aKaibdaya}

\subsection*{$\mathbb{S}^4 \setminus \spun (8a_4)$}

\is{OLLLLALvvQALLvwzPMQQPAQAMzQMQQPQQQcccegdihormipktprrCApvxsEzvEEvvGGAwyyHJBLCIDKGMDLJIFHHILMMLNNNaaRaRaaagarb2aVbObDaga0aRaVbDa6aybWaNbPbbbRararbRayauaAbgagaaaEaEaPboayaaaGbVbObfbfaoa+aPbJbPbkacb2afaJagaPb2aebIatb8agbfa}

\subsection*{$\mathbb{S}^4 \setminus \spun (8a_9)$}

\is{OLLLLLvAvvAvzwPLAQPPPQwAAQQQQQQQQQcbcchgijotzzrrmsCAvxpHCAGEBGDzuHvxwLJKNFyDIBMBCEEGKFLKNINJMNMMWa1a2aaaaaPbKaQaaaobaaEaEabaaaWbPaPbKayaaayajaFa2atb2aBbybTbjbCa2aPbaaPasbsb1aca1a0bmb-a2abayayaLbwbbaAbMbBayawaya6aMb2aTb}

\subsection*{$\mathbb{S}^4 \setminus \spun (8a_7)$}

\is{OLLLLAPLLzAAPMzMPvMLwQQMPQwQQLQMQQccbgfhhiikgnjnmrsppsrpxtvyzCCuDwvCwFEIFBAIGJGHEDJFJLLMNMMKLNMNaagarbaaaaaaaaaaCbgaUaVaaa-a7aaalata1aTbPbEaPbmarbSb-a7agaZb9aoayaba6a6alaPbub+arb+aWaibNbSaqbxagaMaaaCaCaJbCaca2a+ataGaya}

\subsection*{$\mathbb{S}^4 \setminus \spun (8a_3)$}

\is{QLLLALLvwwzQwLLPQAALAAAPQPPAwQPQQQQQbdgdfefiittlksmnupBAqqvCvzsyyuAByEIJyFIIBJDMNEOKGHHNKLJLLMMNPPOP7a5adapbSb2aPbTbpbEaaaba2aebPbPb5aMbaaaaCbyaXbGbvaSbgababa2atasbub4aaaAa+a1bkakaNakaqbaaaa2aVbibpbYaoasbbanb+aNbUaVb9aqbca2avava}

\subsection*{$\mathbb{S}^4 \setminus \spun (8a_5)$}

\is{QLLLAvLvMzwPMPMzQMvAQMPzPMQzAQAQQQQQabfdfhoonjsnrurusvoAvtAqCDuAtvFBFxFEyAzJEJKIIKOFHHJOPMKOPNLONMNPfafaMaoaMaBaMaMaaaRaaatataIbBaIbMbIboa7avaJaLaraaahafbHaJaTbabyaCawaRaJagaRafaaaraKbladatatahaoa1aTb0aVbaayawaKapbJagasaqbfafa2a}

\subsection*{$\mathbb{S}^4 \setminus \spun (8a_1)$}

\is{QLLLvLLPLwvLPQQPzAMzAzAQMMMPQQMQAQQQabcgfhhlpslxqpupqtnuqrrqyxuxCGHEEGDDDJAMLHCLMMLKJONFOKOMIPJKONPPcacacayavayayaYaDbmaKaaasbMaaaSbCaNatbJaHbyayasava+aKaoaaaaaaaiamaQaZaZanbwacasb3awawaHaHaHaMavayaaaUacajbgbIaHacahbcacaQbqbJaoa}

\subsection*{$\mathbb{S}^4 \setminus \spun (8a_2)$}

\is{ULLLLwvQLzvwzLwAMLQvQAMQAPzQQzQQAPQQQQQccbfghkjfhojvpynzuoqGFzuHIJuKvICCxxLyNNMAGJBLHRRKNEONSTSJKRTQLMPTPPSQRTEaga7aTaaaaaKagaMb3aqbaa7aaasaaawbfaTbrbPaga2baaAazaFbMarazbGajayaGaqaFaLa0aaaga7aebvaZa8atadaEavafagaxbaaqa3aFb2ambZbIbyagaGadagasauavacasa}

\subsection*{$\mathbb{S}^4 \setminus \spun (8a_{15})$}

\is{SLLLLALwvvQLPPAzQvvQMQQPQLzQAPMAQPQQQkabchggfkookvvltmnrsuzqstEGzzBxAxwwBCACDJJFDGJMGIHNOONLKPMRRQMPOPRRQcacacaaaTaTavanbTaTanbdaaatbtbgbKabafaaaiaPbKaYa4aaahbhbYaJaXbFaQbca9a9aKa2bfadarbPbraPbDbJbKaZa2akb4aCbcbXb2avaYaiadafacaqbcaqb0aQaoa}

\subsection*{$\mathbb{S}^4 \setminus \spun (8a_{14})$}

\is{YLLLLvvvPMPPLwMvAvLQvLPzAQQQQPQAQAQMQQMQQQkcbgfmgksonurttwzvDIzArNMwuHEyMOxLAPIQQBQBDPETMFHFTKGKUJQWSUNOOWPUWURVXXXVTVXEagaHaaaPbgaibHaHboaaatauaiataaajbmatatataHblaEarayaHboaga7atava7afa7avaiaiajauagaibObpbEaSbXbgaoaiaZbfaZbxaHbuaaajbVagaHbgaJavaMbHbMbyacbhbdbdbcbgafazb}

\subsection*{$\mathbb{S}^4 \setminus \spun (8a_{12})$}

\is{WLLLLAvvLvvPwPMQQvPLzPvQQMMMQQQAzQQPQQQQQcbgfhhgttnrzsmAtwoCsExFsuCBNHGwOxEyIMAMRDKLGJDPGHJQJTSJLRVRMSVOTVPTSQUTUVaagarbaaaaaagaaaaabaaaaalasaaamaaamaaamaiaibaaPbbarabaaaAaTasaaasa4agaaalasamaqamahahahbKboaka1b1b+agbfaobiaubgajahaoagaObmasaKblagama0bgahbKaPbUb}

\subsection*{$\mathbb{S}^4 \setminus \spun (8n_1)$}

\is{yLLLLvLvQwAQLAQQPQQQQcbgfhhmgirmtslrtnrwsvqpqpruqutxwwxwvxaagarbaanbPbPbgaGaAbVaaaAavaaalagbzbLarbjbAbAbAbnbzbHbga2aKaPbgayanayafaga}

\subsection*{$\mathbb{S}^4 \setminus \spun (8n_2)$}

\is{GLLLLLvAPvvPAPMAQQwQPPQQMQQkbcchgikmlpsjwvowosnyvtpqsszrxBCuBEBABCDAADEEFCFFFWa1a2aaaaaqbFaRa1ahbaabaaaZbDbZbDb4a2aaaqb1ayabaqbSacabafbjbaaba3baapafbpanbaafbRaLaAaAaaabavavaCb}

\subsection*{$\mathbb{S}^4 \setminus \spun (8n_3)$}

\is{mLLLMMMPMQQcecegfiihhjjillkkllZbXbXaFaaaQaaaabraraRaRa6axaaaDa3bgaBb}
	
\end{document}